\documentclass{amsart}

\usepackage{graphicx}% Include figure files
\usepackage{dcolumn}% Align table columns on decimal point
\usepackage{bm}% bold math
\usepackage{amsmath,amsthm,amssymb,stmaryrd}
\usepackage[cmtip,all]{xy}
\newcommand{\longsquiggly}{\xymatrix{{}\ar@{~>}[r]&{}}}
\numberwithin{equation}{section}

\newtheorem{theorem}{Theorem}
\newtheorem{lemma}{Lemma}
\newtheorem{corollary}{Corollary}
\newtheorem{proposition}{Proposition}
\newtheorem{definition}{Definition}

\newcommand{\bbint}[2]{\ensuremath{\;\backslash\!\!\!\!\backslash\!\!\!\!\!\int_{#1}^{#2}}}

\usepackage{accents}
\newlength{\dhatheight}

\usepackage{amsaddr}

\begin{document}
% FOR AMS CLASS
	\title[Zeta function]{Finite-part integral representation of the Riemann zeta function at odd positive integers and consequent representations}
	\author{Eric A. Galapon}
	\address{Theoretical Physics Group, National Institute of Physics, University of the Philippines, Diliman Quezon City, 1101 Philippines}
	\email{eagalapon@up.edu.ph}
	%\date{\today}
	%\subjclass[2000]{XXX}

%FOR ARTICLE CLASS
%\title{{\bf Regularized Limit, Analytic Continuation and Finite-Part Integration}}
%
%\author{Eric A. Galapon \thanks{email: eagalapon@up.edu.ph}\\Theoretical Physics Group, National Institute of Physics\\University of the Philippines, Diliman Quezon City\\1101 Philippines}
%
%\date{\today}

\begin{abstract} The values of the Riemann zeta function at odd positive integers, $\zeta(2n+1)$, are shown to admit a representation proportional to the finite-part of the divergent integral $\int_0^{\infty} t^{-2n-1} \operatorname{csch}t\,\mathrm{d}t$. Integral representations for $\zeta(2n+1)$ are then deduced from the finite-part integral representation. Certain relations between $\zeta(2n+1)$ and $\zeta'(2n+1)$ are likewise deduced, from which integral representations for $\zeta'(2n+1)$ are obtained. 
\end{abstract}
\maketitle

\section{Introduction}
The Riemann zeta function, $\zeta(s)$, extends in the entire complex plane the sum of the reciprocal of the integer powers of the positive integers. The zeta function is defined by the infinite series 
\begin{equation}\label{zetafunc}
	\zeta(s)=\sum_{k=1}^{\infty} \frac{1}{k^s}
\end{equation}
in the domain $\operatorname{Re} s>1$, and by analytic continuation %of the series \eqref{zetafunc} 
in the rest of the complex plane. The values and the arithmetic nature of $\zeta(s)$ at integer arguments are almost known \cite{roman}. Its values at non-positive integers are given by
\begin{equation}
    \zeta(-n)=(-1)^n \frac{B_{n+1}}{n+1}, \;\;\; n=0,1,\,2,\dots ,
\end{equation}
where the $B_{n+1}$'s are the Bernoulli numbers; specifically, $\zeta(0)=-1/2$ and $\zeta(-n)=0$ for even negative integer $-n$. The values at positive even integers are likewise known and are given by
\begin{equation}\label{zetaeven}
    \zeta(2n)=\frac{(-1)^{n+1} (2\pi)^{2n} B_{2n}}{2 (2n)!} , \;\;\; n=1,\,2,\,3\dots ,
\end{equation}
where the $B_{2n}$'s are again the Bernoulli numbers. The Bernoulli numbers are rational being finite sums of rational numbers whose explicit forms are known \cite{gould}. Hence the values and the arithmetic nature of the zeta function at non-positive and even positive integers are settled---the zeta function is rational at non-positive integers; and it is both irrational and transcendental at even positive integers, owing to $\pi$ being both irrational and transcendental.  

However, the values of the zeta function at odd positive integers, the sums
\begin{equation}\label{oddzetasum}
    \zeta(2n+1)=\sum_{k=1}^{\infty}\frac{1}{k^{2n+1}},\;\;\; n=1, 2, 3,\dots ,
\end{equation}
have no known closed form evaluation similar to  the values at non-odd positive integers. Attempts have been made since Euler to evaluate the sums in terms of known mathematical constants and elementary functions without success. But there  is no established fact that even remotely imply that the sum \eqref{oddzetasum} cannot be closed. Not only that the values at odd positive integers are not known in closed form but also their arithmetic nature is a mystery. Whether they are rational or irrational is not known except for $\zeta(3)$ which Apery showed in 1979 to be irrational \cite{apery}. Beyond Apery's result are the result of Rivoal in 2000 that showed that there are infinitely many irrational values of the zeta function at odd positive integers \cite{rivoal}, and the result of Zudilin in 2006 that showed that at least one of the values $\zeta(5)$, $\zeta(7)$, $\zeta(9)$ and $\zeta(11)$ is irrational \cite{zudilin}. These results constitute the best knowledge available today on  the irrationality of the $\zeta(2n+1)$'s. Not only that there is little insight into their irrationality but nothing is known about their transcendence \cite{murty}. 

As the sum \eqref{oddzetasum} does not offer much insight into the nature of $\zeta(2n+1)$, the properties of $\zeta(2n+1)$ can only be accessed by recasting the sum in some other forms or in certain representations that are amenable to analysis. Apery managed to show that $\zeta(3)$ is irrational by recasting the defining infinite series into the sufficiently fast converging series, 
\begin{equation}\label{apery}
    \zeta(3)=\frac{5}{2}\sum_{k=1}^{\infty}\frac{(-1)^{k-1}}{k^3 \binom{2k}{k}} .
\end{equation}
Beukers later improved on the proof of Apery by means of the double integral representation \cite{beukers}
\begin{equation}\label{beukers}
\zeta(3)=\frac{1}{2}\int_0^1\int_0^1\frac{\ln xy}{xy-1}\mathrm{d}x\mathrm{d}y.
\end{equation}
The examples of Apery and Beukers motivate obtaining new representations of $\zeta(2n+1)$ which may provide starting points in the analysis of its properties. Numerous infinite series \cite{yue,lima,cvi,srivasta} and integral \cite{klinowski,goldstein,petros,ito} representations exist for $\zeta(2n+1)$. Each representation that we know so far, including those that will be reported in this paper, can be seen as a failed attempt at evaluating the sum \eqref{oddzetasum} in closed form. Nevertheless, each representation possesses a set of properties that can be exploited to prove certain properties of $\zeta(2n+1)$. Apery's infinite series representation \eqref{apery} and Beukers' integral representation \eqref{beukers} do not yield closed form evaluation of $\zeta(3)$ but they have the appropriate properties to prove the irrationality of $\zeta(3)$. 

In this paper we offer a novel representation of the zeta function at odd positive integers which may hopefully offer some insight into the properties of $\zeta(2n+1)$. We show that the value of $\zeta(2n+1)$ for all positive integer $n$ is proportional to the finite-part of the divergent integral
\begin{equation}\label{divergentsinh}
	\int_0^{\infty} \frac{\mathrm{d}t}{t^{2n+1}\sinh t},
\end{equation}
where the divergence arises from the non-integrable singularity at the origin. If we denote the finite-part by $\bbint{0}{\infty}$, which we will define precisely in the next Section, 
we establish the following %finite-part integral (FPI) 
representation for $\zeta(2n+1)$ in terms of the finite-part of the divergent integral \eqref{divergentsinh}. 
\begin{theorem}[Main Theorem]\label{maintheorem}
The zeta function at odd positive integers assumes the finite-part integral (FPI) representation
\begin{equation}\label{zetafpirep}
    \zeta(2n+1)=(-1)^{n+1} \frac{(2\pi)^{2n}}{(2^{2n}-1)} \bbint{0}{\infty} \frac{\mathrm{d}t}{t^{2n+1} \sinh t},\;\; n=1, 2, 3,\dots  .
\end{equation}
\end{theorem}
\noindent On its own, this representation may be used to compute $\zeta(2n+1)$ using known quadrature evaluations of the finite-part. But the FPI-representation is more than a computational tool---it is a capsule representation that packs a class of representations arising from numerous possible representations of the finite-part integral. 

A fundamental property of the finite-part integral established in \cite{galapon1} and \cite{galapon2} is that the finite-part (under certain analycity conditions) admits a contour integral representation in the complex plane, which implies that the finite-part is the value of a convergent integral and that it may assume an integral representation on the real line. Here we obtain two integral representations of the finite-part of \eqref{divergentsinh} which in turn give two corresponding integral representations for the zeta function at odd positive integers on the positive real axis. The representations 
lead to the following integral representations for the Apery constant,
\begin{equation}
    \zeta(3)=\frac{4 \pi^2}{3} \int_0^{\infty} \left(\frac{t}{\sinh t}+\frac{t^2}{6}-1\right)\,\frac{\mathrm{d}t}{t^4},
\end{equation}
\begin{equation}
    \zeta(3)=\frac{2\pi^2}{9}\int_0^{\infty}\left(-\frac{6t \cosh^3 t}{\sinh^4 t} + \frac{6 \cosh^2 t}{\sinh^3 t} +\frac{5 t \cosh t}{\sinh^2 t}-\frac{3}{\sinh t}\right)\,\frac{{\rm d}t}{t} .
\end{equation}
Moreover, we obtain a contour integral representation for $\zeta(2n+1)$ following from the fundamental contour integral representation of the finite-part of divergent integrals with non-integrable pole singularity at the origin \cite{galapon2,galapon3,galapon4,galapon5,galapon6}. Then we show that an appropriate deformation of the contour of integration uncovers certain relationships between the values of the zeta function $\zeta(2n+1)$ and its derivative $\zeta'(2n+1)$, which in turn lead to integral representations for $\zeta'(2n+1)$. 

The rest of the paper is organized as follows. In Section-\ref{fpidiscussion}, we provide a brief discussion of finite-part integrals and prove a fundamental property of the finite-part integral that is central to the development to follow. 
In Section-\ref{proofMainTheorem}, we prove the finite-part integral representation \eqref{zetafpirep} of the zeta function at odd positive integers using the result of the previous Section. In Section-\ref{proofTheorem2}, we obtain two integral representations for $\zeta(2n+1)$ from the FPI-representation by means of two ways of extracting the finite-part integral. We then deduce another integral representation from the second integral representation. In Section-\ref{proofContourRep}, we derive the contour integral representation of $\zeta(2n+1)$. In Section-\ref{proofTheorem3}, we obtain the integral representation of the first derivative at odd positive integers following from the contour integral representation of $\zeta(2n+1)$, in the process we obtain relationships between $\zeta(2n+1)$ and $\zeta'(2n+1)$. In Section-\ref{conclusion}, we conclude. 

\section{The Finite-Part Integral}\label{fpidiscussion}
Suppose that $f(t)$ is analytic at the origin, $f(0)\neq 0$, and locally integrable in the interval $[0,a]$, where $0<a\leq\infty$. Then the integral
\begin{equation}\label{divergentint}
    \int_0^{a} \frac{f(t)}{t^{\lambda}}\,\mathrm{d}t,\;\;\;\operatorname{Re}\lambda\geq 1,
\end{equation}
is divergent due to the non-integrability of $t^{-\lambda}$ at the origin. The finite-part is obtained by replacing the offending lower limit of integration by some positive $\epsilon<a$ and decomposing the integral in the form
\begin{equation}\label{convint}
    \int_{\epsilon}^{a} \frac{f(t)}{t^{\lambda}}\mathrm{d}t = C_{\epsilon}+D_{\epsilon},
\end{equation}
where $C_{\epsilon}$ is the group of terms that converge in the limit $\epsilon\rightarrow$ and $D_{\epsilon}$ is the group of terms that diverge in the same limit. The finite-part is uniquely defined by requiring that the diverging term $D_{\epsilon}$ consists only of diverging algebraic powers of $\epsilon$ and $\ln\epsilon$ \cite{monegato,hadamard}. 
\begin{definition}
Under the stated conditions on $f(t)$, the finite-part of the divergent integral \eqref{divergentint} is given by
\begin{equation}\label{fpidefinition}
    \bbint{0}{a}\frac{f(t)}{t^{\lambda}}\,\mathrm{d}t = \lim_{\epsilon\rightarrow 0} C_{\epsilon} .
\end{equation}
When the upper limit of integration is infinity, the finite-part is defined to be the limit
\begin{equation}\label{fpidefinitioninf}
    \bbint{0}{\infty}\frac{f(t)}{t^{\lambda}}\,\mathrm{d}t=\lim_{a\rightarrow\infty}\bbint{0}{a}\frac{f(t)}{t^{\lambda}}\,\mathrm{d}t,
\end{equation}
provided the limit exists.
\end{definition}

The following Lemma establishes a fundamental property of the finite-part integral which is the basis of the development to follow. 
\begin{lemma}\label{mainlemma}
Let $f(t)$ satisfy all stated conditions. If $\operatorname{Re}\lambda\geq 1$, then for every positive $b<a$,
\begin{equation}\label{fpilemma}
    \bbint{0}{a} \frac{f(t)}{t^{\lambda}}\,\mathrm{d}t = \bbint{0}{b} \frac{f(t)}{t^{\lambda}}\,\mathrm{d}t + \int_{b}^{a} \frac{f(t)}{t^{\lambda}}\,\mathrm{d}t .
\end{equation}
The equality holds for $a=\infty$ if $f(t) t^{-{\lambda}}$ is integrable at infinity.
\end{lemma}
\begin{proof}
Let $b$ be a positive number such that $\epsilon<b<a$ for some arbitrarily small positive number $\epsilon$. Since the integral in the left hand side of \eqref{convint} is convergent, it is a fundamental theorem in calculus that the integral \eqref{convint} can be split in two parts,
\begin{equation}\label{prefpi}
	\int_{\epsilon}^{a} \frac{f(t)}{t^{\lambda}}\mathrm{d}t = \int_{\epsilon}^{b} \frac{f(t)}{t^{\lambda}}\mathrm{d}t + \int_{b}^{a} \frac{f(t)}{t^{\lambda}}\mathrm{d}t
\end{equation}
Clearly the second integral in \eqref{prefpi} belongs to the converging term $C_{\epsilon}$, and the diverging term $D_{\epsilon}$ solely resides in the first integral. 

To extract the finite-part, we decompose the first integral in terms of its converging and diverging parts,
\begin{equation}
    \int_{\epsilon}^b \frac{f(t)}{t^\lambda}\,\mathrm{d}t = \tilde{C}_{\epsilon} + \tilde{D}_{\epsilon},
\end{equation}
so that the finite-part of the divergent integral $\int_0^b t^{-\lambda} f(t)\,\mathrm{d}t$ is given by
\begin{equation}
    \bbint{0}{b}\frac{f(t)}{t^{\lambda}}\,\mathrm{d}t = \lim_{\epsilon\rightarrow 0} \tilde{C}_{\epsilon} .
\end{equation}
Now the converging part of the original divergent integral is given by
\begin{equation}\label{peko}
    C_{\epsilon}=\tilde{C}_{\epsilon}+\int_b^a \frac{f(t)}{t^{\lambda}}\,\mathrm{d}t.
\end{equation}
Taking the limit as $\epsilon\rightarrow 0$ in both sides of \eqref{peko}, we are led to the desired equality \eqref{fpilemma}, in accordance with the definition of the finite-part integral.

If $t^{-\lambda} f(t)$ is integrable at infinity, then the limit of the finite-part \eqref{fpilemma} as $a\rightarrow\infty$ exists. Thus equation \eqref{fpilemma} holds when the upper limit of integration is infinity. 
\end{proof}

This property of the finite-part integral, which we will refer to as the splitting-property, affords us the ability to extract the finite-part when $f(t)$ has a singularity in the complex plane which is closer to the origin than the upper limit of integration. Choosing $b$ smaller than the distance of the nearest singularity to the origin allows us to perform expansion about $t=0$ and perform term by term integration to identify the diverging terms that have to be discarded. 

\section{The Finite-Part Integral Representation:\\Proof of the Main Theorem}\label{proofMainTheorem}
%\begin{proof}[Proof of the Main Theorem] 
To prove the Main Theorem, we evaluate the finite-part of the divergent integral \eqref{divergentsinh} and show that it is proportional to the zeta function $\zeta(2n+1)$. We cast the divergent integral in the following form to assume the form required in the previous Section,
\begin{equation}
    \int_0^{\infty}\frac{\mathrm{d}t}{t^{2n+1}\sinh t}=\int_0^{\infty}\frac{1}{t^{2n+2}}\left(\frac{t}{\sinh t}\right)\,\mathrm{d}t.
\end{equation}
We identify $f(t)=t/\sinh t$; $f(t)$ is locally integrable in the entire real line, $t^{-2n-1}f(t)$ is integrable at infinity, and $f(0)\neq 0$. Then we can use Lemma-\ref{mainlemma} to evaluate the finite-part. The singularities of $f(t)$ in the complex plane nearest to the origin are at a distance $\pi$. For convenience, we choose $b=1$. Then Lemma-\ref{mainlemma} gives 
\begin{equation}\label{fpirep1}
	\bbint{0}{\infty} \frac{\mathrm{d}t}{t^{2n+1}\sinh t} = \bbint{0}{1} \frac{\mathrm{d}t}{t^{2n+1}\sinh t}+\int_{1}^{\infty} \frac{\mathrm{d}t}{t^{2n+1}\sinh t}.
\end{equation}	
The problem reduces to evaluating a finite-part integral and a regular integral.

To extract the finite-part of the divergent integral $\int_0^1 t^{-2n-1}\operatorname{csch} t\,\mathrm{d}t$, we evaluate the integral
\begin{equation}\label{integralepsilon}
    \int_{\epsilon}^1\frac{\mathrm{d}t}{t^{2n+1}\sinh t},
\end{equation}
for some positive $\epsilon<1$. We do so by employing the expansion
\begin{equation}\label{expansioncsch}
	\frac{1}{\sinh z} = - \sum_{k=0}^{\infty}\frac{(2^{2k}-2) B_{2k}}{(2k)!} z^{2k-1},\;\; |z|<\pi,
\end{equation}
where the $B_{2k}$'s are Bernoulli numbers. Substituting the expansion back into integral \eqref{integralepsilon} and performing term by term integration, which is allowed because the series converges uniformly in the entire range of integration, yield
\begin{equation}\label{buko}
\int_{\epsilon}^{1} \frac{\mathrm{d}t}{t^{2n+1}\sinh t} = - \sum_{k=0}^{\infty} \frac{(2^{2k}-2)B_{2k}}{(2k)! (2k-2n-1)} \left(1-\epsilon^{2k-2n-1}\right).
\end{equation}
Since the $\epsilon$-terms either diverge or vanish in the limit, the finite-part is obtained by dropping all terms in $\epsilon$. Thus the desired finite-part is
\begin{equation}\label{fpi0to1}
	\bbint{0}{1} \frac{\mathrm{d}t}{t^{2n+1}\sinh t} = - \sum_{k=0}^{\infty} \frac{(2^{2n}-2)B_{2k}}{(2k)! (2k-2n-1)}.
\end{equation}

To evaluate the regular integral in the second term of equation \eqref{fpirep1}, we use the expansion
\begin{equation}\label{expansion2}
	\frac{1}{\sinh z} = \frac{1}{z} + 2 z \sum_{k=1}^{\infty} \frac{(-1)^k}{\pi^2 k^2 + z^2},\;\;\; \frac{i z}{\pi}\neq \mathbb{Z}.
\end{equation}
Substituting this expansion back in the integral and performing term by term integration give
\begin{equation}\label{kwa}
	\int_1^{\infty}\frac{\mathrm{d}t}{t^{2n+1}\sinh t} = \frac{1}{(2n+1)} + 2 \sum_{k=1}^{\infty} (-1)^k \int_1^{\infty} \frac{\mathrm{d}t}{t^{2n} (\pi^2 k^2 + t^2)}.
\end{equation}
Term by term integration is allowed because the sum in \eqref{expansion2} converges uniformly in the entire interval of integration and each integral in the right hand side of \eqref{kwa} converges. Using reference \cite[p.80,\#2.176]{gr}, we obtain
\begin{equation}\label{nana}
	\int_1^{\infty}\frac{\mathrm{d}x}{t^{2n} (\pi^2 k^2+t^2)} = \frac{1}{(2n-1) \pi^2 k^2} - \frac{1}{\pi^2 k^2} \int_1^{\infty} \frac{\mathrm{d}t}{t^{2(n-1)} (\pi^2 k^2 + t^2)}
\end{equation}
In this form, it is clear that we can evaluate the desired integral by $n$ recursive application of this identity. After $n$ recursions, we arrive at the tail with the integral
\begin{equation}
	\int_1^{\infty} \frac{\mathrm{d}t}{\pi^2 k^2 + t^2} = \frac{1}{2k}-\frac{1}{k\pi} \arctan\!\left(\frac{1}{k\pi}\right).
\end{equation}
Then by induction, the integral in the left hand side of \eqref{nana} evaluates to
\begin{equation}\label{bec}
	\begin{split}
	\int_1^{\infty} \frac{\mathrm{d}t}{t^{2n}(\pi^2 k^2 + t^2)}=& \sum_{l=1}^{n} \frac{(-1)^l}{(2l-2n-1) \pi^{2l} k^{2l}} \\
	& + \frac{(-1)^n}{2\pi^{2n} k^{2n+1}} -\frac{(-1)^n}{\pi^{2n+1} k^{2n+1}}\, \arctan\!\left(\frac{1}{k\pi}\right) .
	\end{split}
\end{equation}

Substituting the integral \eqref{bec} back into equation \eqref{kwa} and distributing the summation, we arrive at
\begin{equation}\label{sumsum2}
	\begin{split}
	\int_1^{\infty} \frac{\mathrm{d}t}{t^{2n+1}\sinh t} =& \frac{1}{(2n+1)} + 2 \sum_{l=1}^n \frac{(-1)^l}{(2n+1-l)\pi^{2n}} \sum_{k=1}^{\infty} \frac{(-1)^k}{k^{2l}}\\
	&\hspace{-6mm}+\frac{(-1)^n}{\pi^{2n}}\sum_{k=1}^{\infty}\frac{(-1)^k}{k^{2n+1}}-(-1)^n \frac{2}{\pi^{2n+1}} \sum_{k=1}^{\infty} \frac{(-1)^k}{k^{2n+1}}\arctan\!\left(\frac{1}{k\pi}\right)
	\end{split} .
\end{equation}
The first two infinite series are evaluated  in terms of the zeta function by means of the known sum
\begin{equation}\label{zetasum}
	\sum_{k=1}^{\infty} \frac{(-1)^k}{k^s}=(2^{1-s}-1)\zeta(s), \;\;\; \mathrm{Re} s>0 .
\end{equation}
To evaluate the third summation, we use the expansion
\begin{equation}\label{arctan}
	\arctan z=\sum_{l=0}^{\infty} \frac{(-1)^l}{(2l+1)}z^{2l+1}, \;\;\; |z|<1.
\end{equation}
This expansion \eqref{arctan}  holds for our present case because $z=1/k\pi<1$ for all $k$. We substitute the expansion, interchange the order of summation, which we can do because the double sum converges absolutely, and perform simplifications. The result is
\begin{equation}\label{gogol}
	\sum_{k=1}^{\infty} \frac{(-1)^k}{k^{2n+1}}\arctan\!\left(\frac{1}{k\pi}\right)=(-1)^n \pi^{2n+1} \sum_{k=n+1}^{\infty} \frac{(-1)^k (2^{2k}-2)}{(2k-2n-1)(2\pi)^{2k}}\zeta(2k) .
\end{equation}
Gathering the sums \eqref{zetasum} and \eqref{gogol} back into equation \eqref{sumsum2}, we find obtain the integral 
\begin{equation}\label{onetoinf}
\begin{split}
	\int_1^{\infty}  \frac{\mathrm{d}t}{t^{2n+1}\sinh t} =&\frac{1}{(2n+1)} - 2 \sum_{k=1}^{\infty} \frac{(-1)^k (2^{2k}-2)}{(2k-2n-1) 2^{2k}}\zeta(2k)\\
	&\hspace{12mm}+ \frac{(-1)^{n+1} (2^{2n}-1)}{(2\pi){2n}} \zeta(2n+1) .
	\end{split}
\end{equation}

Finally, we return the finite-part integral \eqref{fpi0to1} and the regular-integral \eqref{onetoinf} back into equation \eqref{fpirep1}. Writing $\zeta(2k)$ in \eqref{onetoinf} in terms of the Bernoulli numbers by means of equation \eqref{zetaeven}, we find that the first two terms of \eqref{onetoinf} exactly cancels the finite-part integral \eqref{fpi0to1}, leaving us with the expression for the desired finite-part,
\begin{equation}
    \bbint{0}{\infty} \frac{\mathrm{d}t}{t^{2n+1} \sinh t}=\frac{(-1)^{n+1}(2^{2n}-1)}{(2\pi)^{2n}}\,\zeta(2n+1).
\end{equation}
Solving for $\zeta(2n+1)$ yields the finite-part integral representation \eqref{maintheorem} for the value of the zeta function at odd positive integers. This completes the proof.
%\end{proof}

An immediate consequence of our method of proof is a parametrized representation of $\zeta(2n+1)$ obtained by lifting the restriction that we have on the parameter $b$ in splitting the integral. \begin{theorem}
For every positive integer $n$ and for all positive $b<\pi$, the finite-part of the divergent integral $\int_0^{\infty}{\rm d}t/t^{2n+1} \sinh t$ is given by
\begin{equation}\label{fpirep0}
\begin{split}
    \bbint{0}{\infty}\frac{{\rm d}t}{t^{2n+1}\sinh t} = \sum_{k=0}^{\infty} \frac{(2^{2k}-2)B_{2k} b^{2k-2n-1}}{(2k)! (2n+1-2k)}  + \int_b^{\infty}\frac{{\rm d}t}{t^{2n+1}\sinh t},
    \end{split}
\end{equation}
and the zeta function at odd positive integers thereby assumes the representation
\begin{equation}\label{zetaoddrep0}
    \zeta(2n+1)=(-1)^{n+1} \frac{(2\pi)^{2n}}{(2^{2n}-1)} \left(\sum_{k=0}^{\infty} \frac{(2^{2k}-2)B_{2k} b^{2k-2n-1}}{(2k)! (2n+1-2k)}  + \int_b^{\infty}\frac{{\rm d}t}{t^{2n+1}\sinh t}\right) 
\end{equation}
\end{theorem}
\begin{proof} 
By Lemma-\ref{mainlemma}, the finite-part of the divergent integral \eqref{divergentsinh} can be expressed as
    \begin{equation}\label{fpirep11}
	\bbint{0}{\infty} \frac{\mathrm{d}t}{t^{2n+1}\sinh t} = \bbint{0}{b}\frac{\mathrm{d}t}{t^{2n+1}\sinh t}+\int_{b}^{\infty} \frac{\mathrm{d}t}{t^{2n+1}\sinh t},
\end{equation}	
for any positive number $b$. If $b<\pi$, then we can replace $1/\sinh t$ in the first term of \eqref{fpirep11} with its expansion \eqref{expansioncsch} to allow us to perform term-by-term integration and obtain the desired finite-part,
\begin{equation}\label{fpilesspi}
    \bbint{0}{b}\frac{{\rm d}t}{t^{2n+1}\sinh t} = \sum_{k=0}^{\infty} \frac{(2^{2k}-2)B_{2k} b^{2k-2n-1}}{(2k)! (2n+1-2k)},\;\; 0<b<\pi .
\end{equation}
Substituting \eqref{fpilesspi} back into \eqref{fpirep11} gives the finite-part integral \eqref{fpirep0}. Finally, introducing the finite-part \eqref{fpirep0} back into the Main Theorem, we obtain the representation \eqref{zetaoddrep0}. 
\end{proof}

\section{Integral Representations on the Real Axis}\label{proofTheorem2}
The finite-part of a divergent integral can be expressed as a convergent integral which is a consequence of the fact that it can be represented as a contour integral in the complex plane. However, in this Section, we derive integral representations of the finite-part of the divergent integral \eqref{divergentsinh}  entirely in the real line (without the use of complex analysis) and, consequently, 
produce integral representations for $\zeta(2n+1)$ along the positive real axis. We extract the finite-part by employing two methods and obtain two corresponding integral representations for $\zeta(2n+1)$.

The first integral representation for $\zeta(2n+1)$ arises from extracting the finite-part using the splitting-property of the finite-part integral established in Lemma-\ref{mainlemma}.
\begin{theorem}
For every positive integer $n$, the finite-part of the divergent integral $\int_0^{\infty}{\rm d}t/t^{2n+1} \sinh t$ is given by
    \begin{equation}\label{hanzan2}
\begin{split}
    \bbint{0}{\infty} \frac{\mathrm{d}t}{t^{2n+1}\sinh t} =\int_0^{\infty}\frac{1}{t^{2n+1}} \left(\frac{1}{\sinh t}+\sum_{k=0}^n \frac{(2^{2k}-2) B_{2k}}{(2k)!} t^{2k-1}\right)\mathrm{d}t ,
    \end{split}
\end{equation}
and %so that 
the zeta function at odd positive integers thereby assumes the integral representation
\begin{equation}\label{fpirepx2}
		\begin{split}
		&\zeta(2n+1) = (-1)^{n+1} \frac{(2\pi)^{2n}}{(2^{2n}-1)} \int_0^{\infty} \left(\frac{1}{\sinh(t)  }+\sum_{k=0}^n \frac{(2^{2k}-2) B_{2k}}{(2k)!} t^{2k-1}\right)\,\frac{\mathrm{d}t}{t^{2n+1}}.
		\end{split}
	\end{equation}
\end{theorem}
\begin{proof}
Specializing the representation \eqref{fpirep0} to $b=1$, the finite-part integral takes the form
\begin{equation}\label{fpirepx1}
    \bbint{0}{\infty} \frac{\mathrm{d}t}{t^{2n+1}\sinh t} = \sum_{k=0}^{\infty} \frac{(2^{2k}-2) B_{2k}}{(2k)! (2n+1-2k)} +  \int_1^{\infty} \frac{\mathrm{d}t}{t^{2n+1}\sinh t}.
\end{equation}
We rewrite the infinite sum in equation \eqref{fpirepx1} by splitting it in two parts,
\begin{equation} \label{split}
\sum_{k=0}^{\infty} \frac{(2^{2k}-2) B_{2k}}{(2k)! (2n+1-2k)} = \sum_{k=0}^{n} \frac{(2^{2k}-2) B_{2k}}{(2k)! (2n+1-2k)}-\sum_{k=n+1}^{\infty} \frac{(2^{2k}-2) B_{2k}}{(2k)! (2k-2n-1)}, 
\end{equation}
with the intention to recast both terms as integrals. 

To recast the finite sum in the first term, we implement the substitution
\begin{equation}
    \frac{1}{(2n+1-2k)} = \int_1^{\infty} \frac{\mathrm{d}t}{t^{2n-2k+2}}, \;\; k=0,\dots n,
\end{equation}
to obtain
\begin{equation}\label{term1}
    \sum_{k=0}^n \frac{(2^{2k}-2) B_{2k}}{(2k)! (2n+1-2k)} = \int_1^{\infty}\sum_{k=0}^n \frac{(2^{2k}-2) B_{2k}}{(2k)!}t^{2k-1}\frac{{\rm d}t}{t^{2n+1}} .
\end{equation}

On the other hand, to rewrite the infinite series in the second term, we make a similar substitution,
\begin{equation}\label{boko}
    \frac{1}{(2k-2n-1)}=\int_0^1 t^{2k-2n-2}\mathrm{d}t,\;\; k=(n+1), (n+2),\dots 
\end{equation}
in the sum. Upon substituting \eqref{boko} back into the second term, we pull out the integral and add a zero to the infinite series so that the second term becomes,
\begin{equation}
\begin{split}
    -\sum_{k=n+1}^{\infty} \frac{(2^{2k}-2) B_{2k}}{(2k)! (2k-2n-1)}=&\int_0^1 \frac{1}{t^{2n+1}} \left[- \sum_{k=n+1}^{\infty} \frac{(2^{2k}-2) B_{2k}}{(2k)!} t^{2k-1}\right.\\
    &\hspace{-18mm}\left.-\sum_{k=0}^n \frac{(2^{2k}-2) B_{2k}}{(2k)!} t^{2k-1}
    +\sum_{k=0}^n \frac{(2^{2k}-2) B_{2k}}{(2k)!} t^{2k-1} \right]\mathrm{d}t .
    \end{split}
\end{equation}
The first two terms inside the brackets in the right side combine to give $1/\sinh t$ in accordance with the expansion \eqref{expansioncsch} so that 
\begin{equation}\label{term2}
\begin{split}
    -\sum_{k=n+1}^{\infty} \frac{(2^{2k}-2) B_{2k}}{(2k)! (2k-2n-1)}=\int_0^1\frac{1}{t^{2n+1}} \left(\frac{1}{\sinh t}+\sum_{k=0}^n \frac{(2^{2k}-2) B_{2k}}{(2k)!} t^{2k-1}\right)\mathrm{d}t .
    \end{split}
\end{equation}

We return \eqref{term1} and \eqref{term2} back into \eqref{split}, and the result back to \eqref{fpirepx2}. We find that all terms combine into a single integral, giving the representation \eqref{hanzan2} for the finite-part integral representation. 

Substituting the representation \eqref{hanzan2}
for the finite-part integral back into the Main Theorem yields the integral representation \eqref{fpirepx2} for the zeta function at odd positive integers.
\end{proof}

The next integral represenation for $\zeta(2n+1)$ arises from extracting the finite-part by means of successive integration by parts. 
\begin{theorem}\label{zetaoddorig}
For every positive integer $n$, the finite-part of the divergent integral $\int_0^{\infty}{\rm d}t/t^{2n+1} \sinh t$ is given by
    \begin{equation}\label{fpider}
		\bbint{0}{\infty} \frac{\mathrm{d}t}{t^{2n+1}\sinh t} = \frac{1}{(2n+1)!} \int_{0}^{\infty} \frac{\mathrm{d}^{2n+1}}{\mathrm{d}t^{2n+1}}\left(\frac{t}{\sinh t}\right)\frac{\mathrm{d}t}{t} .
	\end{equation}
and the zeta function at odd positive integers thereby assumes the integral representation
\begin{equation}\label{zetaoddintegrep1}
\zeta(2n+1) =(-1)^{n+1} \frac{(2\pi)^{2n}}{(2^{2n}-1) (2n+1)!} \int_0^{\infty} \frac{\mathrm{d}^{2n+1}}{\mathrm{d}t^{2n+1}}\left(\frac{t}{\sinh t}\right)\,\frac{\mathrm{d}t}{t} .
\end{equation}
\end{theorem}
\begin{proof} 
    For some positive $\epsilon$, we evaluate the integral
\begin{equation}
	\int_{\epsilon}^{\infty} \frac{\mathrm{d}t}{t^{2n+1}\sinh t} = \int_{\epsilon}^{\infty} \left(\frac{t}{\sinh t}\right) \frac{\mathrm{d}t}{t^{2n+2}} .
\end{equation}
Repeated $2n$-integration by parts yields
\begin{equation}\label{resultintegbyparts}
	\begin{split}
		\int_{\epsilon}^{\infty} \frac{\mathrm{d}t}{t^{2n+1}\sinh t} =& \sum_{k=0}^{2n} \frac{\epsilon^{k-2n-1}}{(2n-k+1) (2n-k+2)_k} \frac{\mathrm{d}^{k}}{\mathrm{d}\epsilon^k}\left(\frac{\epsilon}{\sinh \epsilon}\right)\\
		& + \frac{1}{(2n+1)!} \int_{\epsilon}^{\infty} \frac{\mathrm{d}^{2n+1}}{\mathrm{d}t^{2n+1}}\left(\frac{t}{\sinh t}\right)\frac{\mathrm{d}t}{t},
		\end{split}
\end{equation}
where the boundary term at infinity has vanished. 

We now identify the diverging and converging terms. From the expansion \eqref{expansioncsch}, we have for arbitrarily small $\epsilon$,
\begin{equation}
    \epsilon^k \frac{{\rm d}^k}{{\rm d}\epsilon^k}\left(\frac{\epsilon}{\sinh \epsilon}\right)=\sum_{l=\lceil k/2\rceil}^{\infty}\frac{(2^{2l}-2)B_{2l}}{(2l)!} \epsilon^{2l},\;\; k=0,\dots, 2n ,
\end{equation}
Then each term in the summation in \eqref{resultintegbyparts} either diverges or vanishes as $\epsilon\rightarrow 0$, with the first vanishing term $O(\epsilon)$. On the other hand, the integral converges in the limit because %expansion
\begin{equation}
	\frac{\mathrm{d}^{2n+1}}{\mathrm{d}t^{2n+1}}\left(\frac{t}{\sinh t}\right) = O(t), \;\; t\rightarrow 0,
\end{equation}
which is again a consequence of the expansion \eqref{expansioncsch}. Then the converging part is given by
\begin{equation}
    C_{\epsilon}=\frac{1}{(2n+1)!} \int_{\epsilon}^{\infty} \frac{\mathrm{d}^{2n+1}}{\mathrm{d}t^{2n+1}}\left(\frac{t}{\sinh t}\right)\frac{\mathrm{d}t}{t} + O(\epsilon) .
\end{equation}
From the definition of the finite-part integral \eqref{fpidefinitioninf}, we finally obtain the integral representation \eqref{fpider} for the finite-part integral.

Substituting the representation \eqref{fpider} of the finite-part integral back into Theorem-\ref{maintheorem} yields the integral representation \eqref{zetaoddintegrep1} for the zeta function at odd positive integers.
\end{proof}

\begin{proof}[Remark]Of course the integral representations \eqref{hanzan2} and \eqref{fpider} for the finite-part integral are equal and may be transformed into each other, as with other representations. Representation \eqref{hanzan2} can be obtained from \eqref{fpider} by repeated integration by parts until all derivatives have been integrated away. Similarly, representation \eqref{fpider} from \eqref{hanzan2} can be obtained by repeated integration by parts until the factor $t^{-2n-1}$ has been reduced to $t^{-1}$. Both representations are on equal footing in the sense that they arise directly from the definition of the finite-part, but the insights that they may bring may be different because of the different rates of convergence of the integrals in the two representations. For example, the integrand in \eqref{hanzan2} vanishes algebraically at infinity while the integrand in \eqref{fpider} vanishes exponentially there. The rate of convergence of the infinite series representation of $\zeta(3)$ of Apery played a crucial rule in the proof of the irrationality of $\zeta(3)$. The same may be expected of integral representations that their rate of convergence may play a role in discovering properties of the zeta function at odd positive integers.\end{proof}

Now either \eqref{fpirepx2} or \eqref{zetaoddintegrep1} can be a starting point in obtaining more representations for the zeta function $\zeta(2n+1)$. %with different convergence properties. 
For example, we can use Laplace transform to recast \eqref{zetaoddintegrep1} into the following integral representation. 
\begin{corollary} 
The zeta function at odd positive integers assumes the integral representation
\begin{equation}\label{reprepzeta}   
\begin{split}
&\zeta(2n+1) =\frac{(-1)^{n+1} (2\pi)^{2n}}{(2^{2n}-1)(2n+1)!} \\
&\hspace{7mm}\times\int_0^{\infty} \left[\frac{1}{2}s^{2n+1} \zeta\!\left(2,\frac{s+1}{2}\right)+\sum_{l=0}^n (2^{2l}-2) B_{2l} \, s^{2n-2l}\right]\mathrm{d}s,
\end{split}
\end{equation}
for $n=1, 2,\dots$, where $\zeta(2,z)$ is a particular Hurwitz zeta function.
\end{corollary}
\begin{proof}
In the integral representation \eqref{zetaoddintegrep1}, we make the substitution
    \begin{equation}
        \frac{1}{t}=\int_0^{\infty} e^{-s t}\,\mathrm{d}t,\;\;\; s>0 .
    \end{equation}
    in the integral and interchange the order of integrations. The justification of the interchange is given in the Appendix. The substitution and interchange lead to the equality
    \begin{equation}
        \int_0^{\infty} \frac{\mathrm{d}^{2n+1}}{\mathrm{d}t^{2n+1}}\left(\frac{t}{\sinh t}\right)\,\frac{\mathrm{d}t}{t} = \int_0^{\infty}\int_0^{\infty}\frac{\mathrm{d}^{2n+1}}{\mathrm{d}t^{2n+1}}\left(\frac{t}{\sinh t}\right) e^{-s t}\,\mathrm{d}t\, \mathrm{d}s .
    \end{equation}
    The inner integral is a Laplace transform of the derivative of a function. Using known properties of the Laplace transform, we arrive at
    \begin{equation}
    \begin{split}\int_0^{\infty}\frac{\mathrm{d}^{2n+1}}{\mathrm{d}t^{2n+1}}\left(\frac{t}{\sinh t}\right) e^{-s t}\,\mathrm{d}t = & s^{2n+1}\int_0^{\infty} \frac{t}{\sinh t}\, e^{-s t}\mathrm{d}t \\
    &\hspace{6mm}- \sum_{l=1}^{2n+1} s^{2n+1-l}\left. \frac{\mathrm{d}^{l-1}}{\mathrm{d}t^{l-1}}\left(\frac{t}{\sinh t}\right)\right|_{t=0^+} .
    \end{split}
    \end{equation}
    
The right hand side can now be readily evaluated. The remaining integral in the right hand side is tabulated and is given by \cite[p.389,\#3.552.1]{gr}
    \begin{equation}\label{laplace}
        \int_0^{\infty} \frac{t}{\sinh t}\, e^{-s t}\,\mathrm{d}t =\frac{1}{2} \zeta\!\left(2,\frac{s+1}{2}\right),
    \end{equation}
which is valid for $\operatorname{Re}s>-1$. On the other hand, the derivatives are evaluated by means of the expansion \eqref{expansioncsch}. Clearly the odd derivatives vanish at the origin and only the even derivatives contribute, 
    \begin{equation}\label{derivatives}
        \left.\frac{\mathrm{d}^{2r}}{\mathrm{d}t^{2r}}\left(\frac{t}{\sinh t}\right)\right|_{t=0} = - (2^{2r}-2) B_{2r},\;\;\; r=1, 2, 3, \dots . 
    \end{equation}
Substituting the Laplace transform \eqref{laplace} and the derivatives \eqref{derivatives} back into equation, we obtain the equality
    \begin{equation}\label{reprep}
    \begin{split}
        &\int_0^{\infty} \frac{\mathrm{d}^{2n+1}}{\mathrm{d}t^{2n+1}}\left(\frac{t}{\sinh t}\right)\,\frac{\mathrm{d}t}{t}\\
        &\hspace{16mm}= \int_0^{\infty} \left[\frac{1}{2}s^{2n+1} \zeta\left(2,\frac{s+1}{2}\right)+\sum_{l=0}^n (2^{2l}-2) B_{2l} s^{2n-2l}\right]\mathrm{d}s .
        \end{split}
    \end{equation} 
Substituting \eqref{reprep} back into \eqref{zetaoddintegrep1} yields the integral representation \eqref{reprepzeta}. 
\end{proof}

%%%%%%%%%%%%%%%%%%%%%%%%%%%%%%%%%%%%%%%%%%%%%%%%%%%%%%%%%
%%%%%%%%%%%%%%%%%%%%%%%%%%%%%%%%%%%%%%%%%%%%%%%%%%%%%%%%%
\begin{figure}[t]\label{fig:fpicontour}
    \centering
	\includegraphics[scale=0.35]{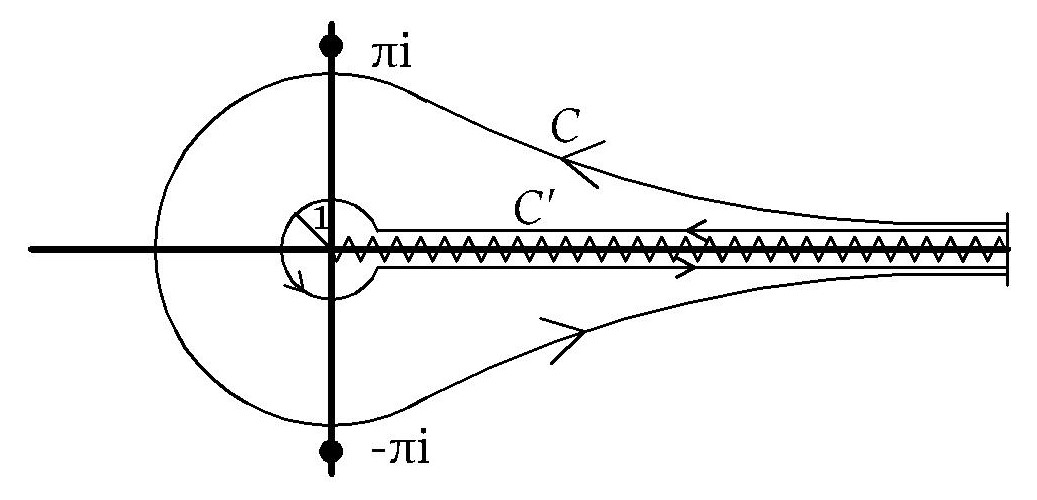}
	\caption{The contour $C$ of integration for the contour integral representation of $\zeta(2n+1)$. The contour $C'$ is a deformation of $C$.}
\end{figure}
%%%%%%%%%%%%%%%%%%%%%%%%%%%%%%%%%%%%%%%%%%%%%%%%%%%%%%%%%
%%%%%%%%%%%%%%%%%%%%%%%%%%%%%%%%%%%%%%%%%%%%%%%%%%%%%%%%%

\section{Contour Integral Representation}\label{proofContourRep}
The finite-part integral \eqref{fpidefinition} admits a contour integral representation which depends on the nature of the non-integrable singularity at origin, whether the singularity is a pole or a branch point. For the divergent integral \eqref{divergentsinh}, we have a pole at the origin and the relevant contour integral representation, adapted from \cite{galapon1} to our present needs, is given by the following result.  
\begin{lemma} Let $f(z)$ be analytic in the positive real line, and let its restriction there be $f(t)$ with $f(0)\neq 0$. If $f(t)t^{-m}$ is integrable at infinity for some positive integer $m$, then the finite-part of the divergent integral $\int_0^{\infty} t^{-m} f(t)\mathrm{d}t$ admits the contour integral representation
\begin{equation}\label{fpicontourrep}
	\bbint{0}{\infty}\frac{f(t)}{t^m}\,\mathrm{d}t = \frac{1}{2\pi i} \int_{\infty}^{(0^+)} \frac{f(z)}{z^m} \left(\log z - i\pi\right)\,\mathrm{d}z,
\end{equation}
where $\log z$ takes the positive real axis as its branch cut, and the contour straddles the branch cut starting from the positive infinity on top of the cut, going around the origin and then going back to the positive infinity below the cut, with the contour not enclosing any of the singularities of $f(z)$. 
\end{lemma}

Thus replacing the finite-part integral with its contour integral representation yields a contour integral representation for the zeta function at odd positive integers.
%%%%%%%%%%%%%%%%%%%%%%%%%%
%%%%%%%%%%%%%%%%%%%%%%%%%%
\begin{figure}[t]\label{fig:deformed1}
    \centering
	\includegraphics[scale=0.4]{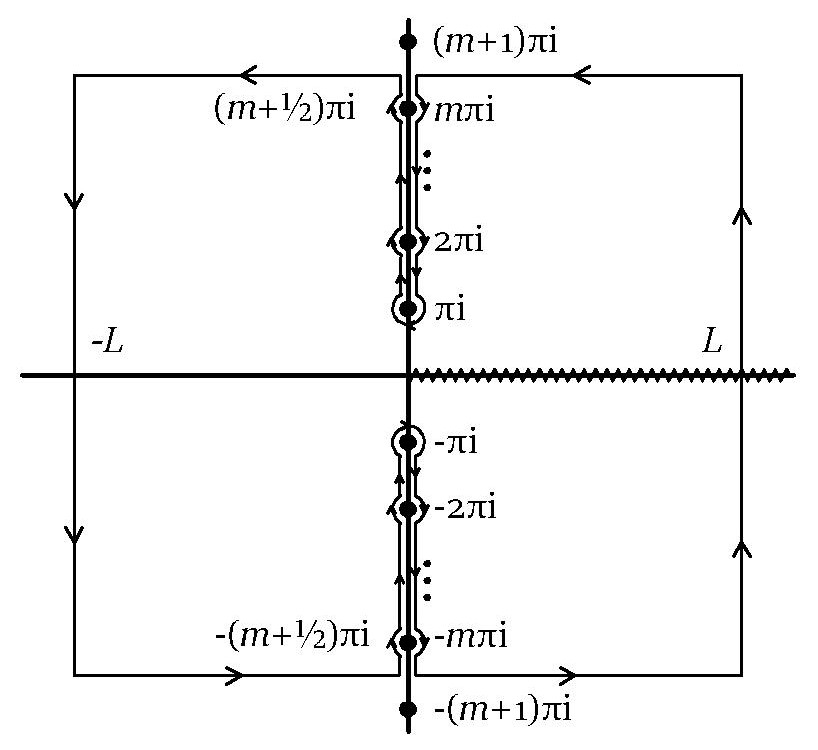}
	\caption{The deformed contour $\Gamma$ of integration to prove the contour integral representation for $\zeta(2n+1)$.}
\end{figure}
%%%%%%%%%%%%%%%%%%%%%%%%%%
%%%%%%%%%%%%%%%%%%%%%%%%%%
\begin{theorem} For every positive integer n,
\begin{equation}\label{zetaoddcrep}
	\zeta(2n+1)=(-1)^{n+1} \frac{(2\pi)^{2n}}{(2^{2n}-2)} \int_{\infty}^{(0+)} \frac{\log z}{z^{2n+1} \sinh z}\,\frac{\mathrm{d}z}{2\pi i} 
\end{equation}
where the contour does not enclose any of the singularities of $z/\sinh z$ which are $z=\pm i\pi, \pm 2i\pi, \pm 3i\pi,\dots$ as shown in Figure-1 denoted by $C$.
\end{theorem}
\begin{proof}
Using the integral representation \eqref{fpicontourrep}, the finite-part integral assumes the representation
\begin{equation}\label{fpicontourrep2}
\begin{split}
    \bbint{0}{\infty}\frac{\mathrm{d}t}{t^{2n+1}\sinh t} = \frac{1}{2\pi i} \int_{\infty}^{(0^+)} \frac{1}{z^{(2n+1)+1}}\left(\frac{z}{\sinh z}\right) \left(\log z - i\pi\right)\,\mathrm{d}z,
    \end{split}
\end{equation}
The contour integration in the right hand side can be distributed. The second term is proportional to the $(2n+1)$ derivative of $z/\sinh z$ at the origin; since $z/\sinh z$ is even, the derivative vanishes. This leads to the contour integral representation for the finite-part integral,
\begin{equation}\label{fpicontourrep3}
\begin{split}
    \bbint{0}{\infty}\frac{\mathrm{d}t}{t^{2n+1}\sinh t} = \frac{1}{2\pi i} \int_{\infty}^{(0^+)} \frac{\log z}{z^{2n+1}\sinh z}\,\mathrm{d}z,
    \end{split}
\end{equation}
Substituting this back into equation \eqref{maintheorem} yields the contour integral representation for the zeta function at positive odd integers \eqref{zetaoddcrep}.

An independent proof can be given for the contour integral representation \eqref{zetaoddcrep}. We deform the contour of integration into the contour $\Gamma$ as shown in Figure-2 and consider the integral around the contour $\Gamma$,
	\begin{equation}
	    \int_{\Gamma} \frac{\log z}{z^{2n+1} \sinh z}\,\frac{\mathrm{d}z}{2\pi i} .
	\end{equation}
	The integrals along the vertical lines, either from the left or from the right, has the bound
	\begin{equation}
		\int_V \frac{\log z}{z^{2n+1} \sinh z}\,\mathrm{d}z < \frac{2^{\frac{1}{2}}}{(\cosh(2L)-1)^{\frac{1}{2}}} \int_0^{\infty} \frac{\left[\ln(L^2+y^2)+3\pi\right]}{\left(L^2 + y^2\right)^{3/2}}\,\mathrm{d}y ,
	\end{equation}
uniformly for all $m$ and $n$. Clearly the right hand side vanishes in the limit $L\rightarrow \infty$. On the other hand, the integrals along the horizontal lines, either above or below, has the bound
\begin{equation}
	\int_H \frac{\log z}{z^{2n+1}\sinh z}\,\mathrm{d}z < \frac{1}{\pi^3 (m+1/2)^3} \int_0^{\infty} \frac{\left[\ln(x^2+(m+1/2)^2 \pi^2)\right]}{\cosh x}\,\mathrm{d}x, 
\end{equation} 
uniformly for all $L$ and $n$. This likewise vanishes in the limit $m\rightarrow\infty$. 

The overlapping integrals along the imaginary axis cancel out and we are left with the contributions coming from the paired semi-circles which combine to form circles around the poles of $\operatorname{csch} z$. The resulting integral around each circle is proportional to the residue at the pole. Taking everything into account, the result is
\begin{equation}\label{integralaround}
	\int_{\Gamma} \frac{\log z}{z^{2n+1} \sinh z}\,\frac{\mathrm{d}z}{2\pi i}=-\sum_{m=-\infty}^{\infty}\!\!\!\!{}^{\prime} \operatorname{Res}\left[\frac{\log z}{z^{2n+1}\sinh z},z=m\pi i\right],
\end{equation} 
where the prime indicates that $m=0$ is excluded from the sum and the negative sign arises from the fact that the circles are traversed in the clockwise direction. The residues at $z=i m\pi=m\pi e^{i\pi/2}$ are calculated to be
\begin{equation}\label{residueup}
	\operatorname{Res}\left[\frac{\log z}{z^{2n+1}\sinh z},z=i m\pi\right] = -\frac{i (-1)^{m+n}}{\pi^{2n+1}m^{2n+1}}\,\left(\ln\pi m +\frac{\pi i}{2}\right);
\end{equation}
on the other hand, the residues at $z=-i m\pi = m\pi e^{i3\pi/2}$ are calculated to be
\begin{equation}\label{residuedown}
	\operatorname{Res}\left[\frac{\log z}{z^{2n+1}\sinh z},z=-i m\pi\right] = \frac{i (-1)^{m+n}}{\pi^{2n+1}m^{2n+1}}\,\left(\ln\pi m +\frac{3\pi i}{2}\right). 
\end{equation}
Remember that the branch cut of $\log z$ has been chosen to be the positive real axis so that the argument of $-i$ is $3\pi/2$.

Substituting these residues back into equation-\eqref{integralaround}, the logarithmic terms cancel and we arrive at the expression
\begin{equation}
	\int_{\Gamma} \frac{\log z}{z^{2n+1} \sinh z}\,\frac{\mathrm{d}z}{2\pi i} = \frac{(-1)^n}{\pi^{2n}} \sum_{m=1}^{\infty} \frac{(-1)^{m}}{m^{2n+1}} .
\end{equation}
Since the integrals $\int_{\infty}^{(0^+)}$ and $\int_{\Gamma}$ are equal, we obtain 
\begin{equation}
	\int_{\infty}^{(0^+)} \frac{\log z}{z^{2n+1} \sinh z}\,\frac{\mathrm{d}z}{2\pi i} = \frac{(-1)^{n+1} (2^{2n}-1)}{(2\pi)^{2n}} \, \zeta(2n+1) ,
\end{equation}
 upon using the summation formula \eqref{zetasum}. Then the contour integral representation \eqref{zetaoddcrep} follows.
\end{proof}

The contour integral representation allows us to obtain more integral representations by appropriate deformations of the contour of integration. In particular, we can arrive at the representation \eqref{fpirepx1} by deforming the contour of integration into the key-hole contour $C'$ shown in Figure-1. Doing so leads to
\begin{equation}\label{lae}
    \begin{split}
        \int_{C} \frac{\log z}{z^{2n+1} \sinh z}\,\frac{\mathrm{d}z}{2\pi i}=\int_{C_1} \frac{\log z}{z^{2n+1} \sinh z}\,\frac{\mathrm{d}z}{2\pi i} + \int_1^{\infty}\frac{\mathrm{d}t}{t^{2n+1} \sinh t} .
    \end{split}
\end{equation}
The integral around the circle is evaluated by inserting the expansion \eqref{expansioncsch} for $1/\sinh z$ in the integral and using the parametrization $z=e^{i\theta}$. The result is
\begin{equation}
    \int_{C_1} \frac{\log z}{z^{2n+1} \sinh z}\,\frac{\mathrm{d}z}{2\pi i}= \sum_{k=0}^{\infty} \frac{(2^{2k}-2) B_{2k}}{(2k)! (2n+1-2k)} .
\end{equation}
Substituting this back into equation \eqref{lae}, we reproduce the representation of the finite-part integral given by equation-\eqref{fpirepx1}. 

%%%%%%%%%%%%%%%%%%%%%%%%%%%%%%%%%%%%%%%%%%%%%%%%%%%%
%%%%%%%%%%%%%%%%%%%%%%%%%%%%%%%%%%%%%%%%%%%%%%%%%%%%
\begin{figure}\label{fig:deformed2}
    \centering
	\includegraphics[scale=0.3]{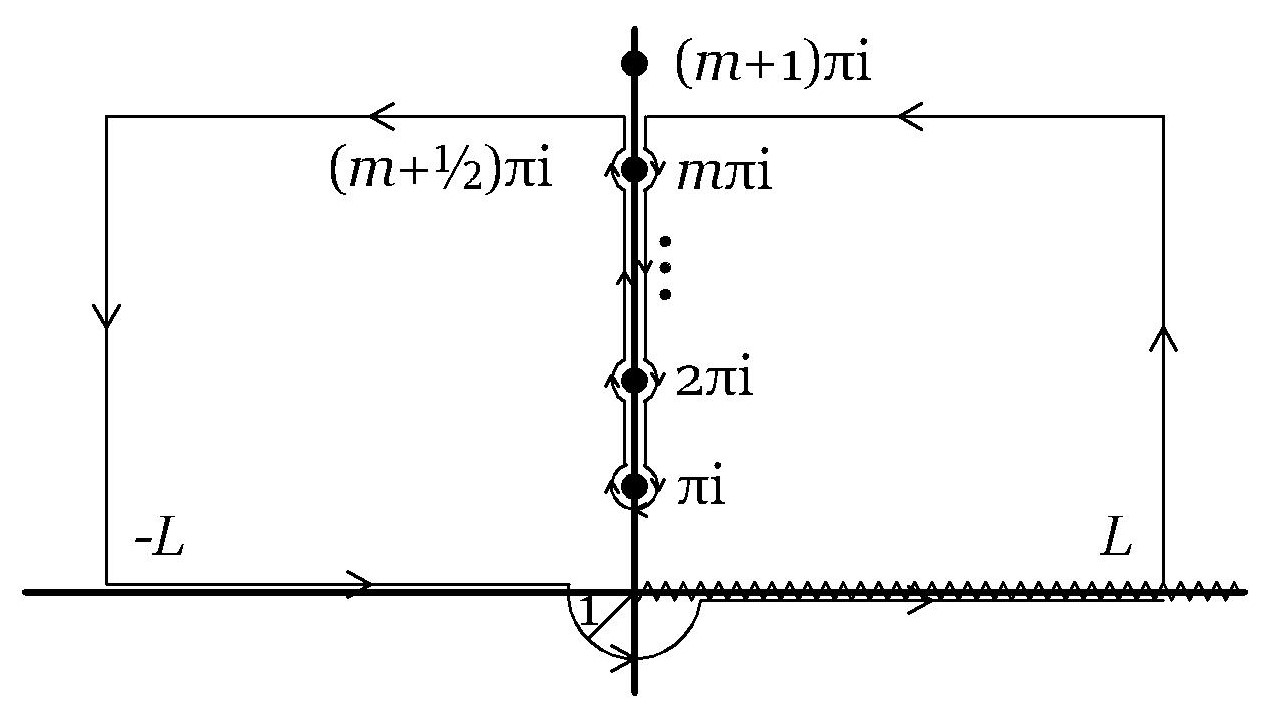}
	\caption{The contour of integration in the derivation of the relationship between $\zeta(2n+1)$ and $\zeta'(2n+1)$.}
\end{figure}
%%%%%%%%%%%%%%%%%%%%%%%%%%%%%%%%%%%%%%%%%%%%%%%%%%%
%%%%%%%%%%%%%%%%%%%%%%%%%%%%%%%%%%%%%%%%%%%%%%%%%%%

\section{The Derivative at Odd Positive Positive Integers}\label{proofTheorem3}
We now obtain a specific consequence of the contour integral representation of $\zeta(2n+1)$.
Notice that the residues contain a logarithmic term which can potentially induce the first derivative of the zeta function at positive odd integers. However, this has canceled with the chose contour of integration. Here we deform the contour in such as way that the logarithmic term does not get canceled. One such deformation is shown in Figure-3. This leads to a relationship between the zeta function and its derivative at positive odd integers in the following results.
\begin{proposition} For all positive integer $n$,
    \begin{equation}\label{zetaprime2}
    \begin{split}
        \zeta'(2n+1) =& \ln\!\left(\frac{\pi}{2^{1/(2^{2n}-1)}}\right) \zeta(2n+1) \\
        &\hspace{-8mm}+\frac{(-1)^n (2\pi)^{2n}}{(2^{2n}-1)}\left[\sum_{k=0}^{\infty}\frac{(2^{2k}-2) B_{2k}}{(2k)! (2n-2k+1)^2} + \int_1^{\infty} \frac{\ln t}{t^{2n+1} \sinh t}\,\mathrm{d}t\right].
        \end{split}
    \end{equation}
\end{proposition}
\begin{proof}
To prove we use the contour integral representation \eqref{zetaoddcrep} of $\zeta(2n+1)$, and deform the contour of integration as indicated in Figure-3. From our results above, the integrals along the sides vanish in the limit $L\rightarrow\infty$ and so with the top integral in the limit $m\rightarrow\infty$. Also the overlapping integrals along the imaginary axis cancel. The only contributing integrals come from the circles around each pole along the positive imaginary axis, along the real axis and along the semicircle. The integrals along the circles are residues and their values are given by equation-\eqref{residueup}.

Performing the integral along the semicircle with the parametrization $z=e^{i\theta}$, $\pi<\theta<2\pi$, and substituting the relevant residues, we obtain
    \begin{equation}\label{beo}
    \begin{split}
        \frac{1}{2\pi i} \int_C \frac{\log z}{z^{2n+1} \sinh z}\,\mathrm{d}z =& \left[\frac{(-1)^{n+1}}{2 \pi^{2n}} \sum_{m=1}^{\infty} \frac{(-1)^m}{m^{2n+1}}+\frac{3}{2}\sum_{k=0}^{\infty}\frac{(2^{2k}-2)B_{2k}}{(2k)!(2n-2k+1)}\right.\\
        &\left.+ \frac{3}{2} \int_1^{\infty} \frac{1}{t^{2n+1}\sinh t}\, \mathrm{d}t\right]\\
        &\hspace{-18mm}+ i \left[\frac{(-1)^n}{\pi^{2n+1}} \sum_{m=1}^{\infty}\frac{(-1)^m \ln(\pi m)}{m^{2n+1}}-\frac{1}{\pi} \sum_{k=0}^{\infty} \frac{(2^{2k}-2)B_{2k}}{(2k)! (2n-2k+1)^2}\right.\\
        &\left.-\frac{1}{\pi} \int_1^{\infty}\frac{\ln t}{t^{2n+1}\sinh t}\, \mathrm{d}t\right] .
        \end{split}
    \end{equation}
    
The left hand side of \eqref{beo} is real and is proportional to $\zeta(2n+1)$. We equate the real and imaginary parts of the two sides of \eqref{beo}. What is relevant to us is the imaginary part. The imaginary part of the left hand side is zero and this yields the equality
    \begin{equation}\label{imaginary}
        \begin{split}
            &\frac{(-1)^n}{\pi^{2n+1}} \sum_{m=1}^{\infty}\frac{(-1)^m \ln(\pi m)}{m^{2n+1}}-\frac{1}{\pi} \sum_{k=0}^{\infty} \frac{(2^{2k}-2)B_{2k}}{(2k)! (2n-2k+1)^2}\\
        &\hspace{34mm}-\frac{1}{\pi} \int_1^{\infty}\frac{\ln t}{t^{2n+1}\sinh t}\, \mathrm{d}t = 0 .
        \end{split}
    \end{equation}
    The first series is evaluated by expanding the logarithm and distributing the sum to yield
    \begin{equation}\label{sumsum}
        \sum_{m=1}^{\infty} \frac{(-1)^m \ln(\pi m)}{m^{2n+1}}= \left(\frac{\ln 2}{2^{2n}}-\frac{(2^{2n}-1)}{2^{2n}}\right) \zeta(2n+1) + \frac{(2^{2n}-1)}{2^{2n}}\zeta'(2n+1),
    \end{equation}
where the first term comes from the sum \eqref{zetasum} and the second sum from the known sum
    \begin{equation}
        \sum_{m=1}^{\infty}\frac{(-1)^m \ln m}{m^{2n+1}} = \frac{(2^{2n-1}-1)}{2^{2n}} \zeta'(2n+1) .
    \end{equation}
Substituting \eqref{sumsum} back into equation \eqref{imaginary} and performing simplifications give the representation \eqref{zetaprime2}.
\end{proof}

In the following two results, we obtain relationships between $\zeta(2n+1)$ and $\zeta'(2n+1)$, and integral representations for $\zeta'(2n+1)$ corresponding to the integral representations \eqref{fpirepx2} and \eqref{zetaoddintegrep1} of $\zeta(2n+1)$.
\begin{theorem}
For all positive integer $n$, the zeta function and its derivative at odd positive integers are related according to the expression
\begin{equation}\label{relationship2}
	\begin{split}
	&\zeta'(2n+1) = \ln\!\left(\frac{\pi}{2^{1/(2^{2n}-1)}}\right) \zeta(2n+1)\\
	 &\hspace{10mm}+ \frac{(-1)^n (2\pi)^{2n}}{(2^{2n}-1)} \int_0^{\infty} \frac{\ln t}{t^{2n+1}}\left(\frac{1}{\sinh t}+\sum_{k=0}^n \frac{(2^{2k}-2) B_{2k} t^{2k-1}}{(2k)!}\right)\mathrm{d}t .
	\end{split}
\end{equation}
Moreover, the derivative at odd positive integers assumes the integral representation
\begin{equation}\label{zetaprimerep3}
\begin{split}
    \zeta'(2n+1)=&\frac{(-1)^n (2\pi)^{2n}}{(2^{2n}-1)} \int_0^{\infty} \ln\!\left(\frac{2^{1/(2^{2n}-1)}}{\pi} t\right)\\
    &\hspace{20mm}\times\left(\frac{1}{\sinh t}+\sum_{k=0}^n \frac{(2^{2k}-2) B_{2k} t^{2k-1}}{(2k)!}\right)\frac{\mathrm{d}t}{t^{2n+1}} .
    \end{split}
\end{equation}
\end{theorem}
\begin{proof}
    This result follows directly from equation \eqref{zetaprime2}. We break the sum in \eqref{zetaprime2} in two parts,
    \begin{equation}
    \begin{split}
        &\sum_{k=0}^{\infty}\frac{(2^{2k}-2) B_{2k}}{(2k)! (2n-2k+1)^2}\\
        &\hspace{18mm}=\sum_{k=0}^{n}\frac{(2^{2k}-2) B_{2k}}{(2k)! (2n-2k+1)^2}+\sum_{k=n+1}^{\infty}\frac{(2^{2k}-2) B_{2k}}{(2k)! (2n-2k+1)^2} .
        \end{split}
    \end{equation}
    We rewrite the second term with the replacement
    \begin{equation}\label{replacement}
        \frac{1}{(2n-2k+1)^2}=-\int_0^1 t^{2k-2n-2}\ln t \, \mathrm{d}t, \;\;\; k-n>\frac{1}{2} .
    \end{equation}
    We have
    \begin{equation}\label{koki}
        \begin{split}
            &\sum_{k=0}^{\infty}\frac{(2^{2k}-2) B_{2k}}{(2k)! (2n-2k+1)^2}=\sum_{k=0}^{n}\frac{(2^{2k}-2) B_{2k}}{(2k)! (2n-2k+1)^2}\\
            &\hspace{8mm}+\int_0^1 \frac{\ln t}{t^{2n+1}}\left(-\sum_{k=0}^n\frac{(2^{2k}-2) B_{2k} t^{2k-1}}{(2k)!}- \sum_{k=n+1}^{\infty} \frac{(2^{2k}-2) B_{2k} t^{2k-1}}{(2k)!}\right.\\
            &\hspace{30mm}\left.+\sum_{k=0}^n\frac{(2^{2k}-2) B_{2k} t^{2k-1}}{(2k)!}\right)\,\mathrm{d}t
        \end{split}
    \end{equation}
    where a zero has been added to the integrand. 
    
    The first two sums in the integral in the left hand side of \eqref{koki} can be combined into a single sum which happens to be equal to $1/\sinh t$; on the other hand, the third sum in the integrand can be integrated using the integral \eqref{replacement}. The result is
    \begin{equation}\label{boxi}
        \begin{split}
            &\sum_{k=0}^{\infty}\frac{(2^{2k}-2) B_{2k}}{(2k)! (2n-2k+1)^2}=\sum_{k=0}^{n}\frac{(2^{2k}-2) B_{2k}}{(2k)! (2n-2k+1)^2}\\
            &\hspace{28mm}+\int_0^1 \frac{\ln t}{t^{2n+1}}\left(\frac{1}{\sinh t}+\sum_{k=0}^n\frac{(2^{2k}-2) B_{2k} t^{2k-1}}{(2k)!}\right)\,\mathrm{d}t .
        \end{split}
    \end{equation}

Next, in the first term of the right hand side of \eqref{boxi}, we perform the substitution
\begin{equation}
    \frac{1}{(2n-2k+1)^2}=\int_1^{\infty} \frac{\ln t}{t^{2n-2k+2}}\,\mathrm{d}t,\;\; k=0,\dots,n, .
\end{equation}
We then substitute the result back in equation \eqref{zetaprime2}. We find that the entire expression combine into a single integral giving us the representation \eqref{relationship2}. 

To arrive at the integral representation \eqref{zetaprimerep3}, we replace $\zeta(2n+1)$ with its integral representation \eqref{fpirepx2} and combine all terms to yield the integral representation \eqref{zetaprimerep3} for the derivative of the zeta function at odd positive integers.
\end{proof}

\begin{theorem}
For all positive integer $n$, the zeta function and its derivative at odd positive integers satisfy the relation
\begin{equation}\label{relationship3}
\begin{split}
    \zeta'(2n+1)=&\left(\ln\left(\frac{\pi}{2^{1/(2^{2n}-1)}}\right)-H_{2n+1}\right) \zeta(2n+1) \\
    &+ (-1)^n \frac{(2\pi)^{2n}}{(2^{2n}-1)(2n+1)!} \int_0^{\infty} \ln t \, \frac{\mathrm{d}^{2n+1}}{\mathrm{d}t^{2n+1}}\left(\frac{t}{\sinh t}\right)\,\frac{\mathrm{{d}t}}{t},
    \end{split}
\end{equation}
where $H_{2n+1}$ is a harmonic number. Moreover, the derivative at odd positive integers assumes the integral representation
\begin{equation}\label{integralrepDzeta}
    \begin{split}
        \zeta'(2n+1)=&(-1)^{n} \frac{(2\pi)^{2n}}{(2^{2n}-1)(2n+1)!}\\
        &\times\int_0^{\infty} \ln\!\left(\frac{2^{1/(2^{2n}-1)} e^{H_{2n+1}}}{ \pi} t\right)\, \frac{\mathrm{d}^{2n+1}}{\mathrm{d}t^{2n+1}}\left(\frac{t}{\sinh t}\right)\,\frac{\mathrm{{d}t}}{t}.
    \end{split}
\end{equation}
\end{theorem}
\begin{proof}
	This representation follows from \eqref{relationship2} by performing successive integration by parts. Let us consider the integral 
	\begin{equation}
	    \int_0^{\infty} \frac{\ln t}{t^{2n+2}} \,R(t)\,\mathrm{d}t
	\end{equation}
	where 
	\begin{equation}\label{rt}
    R(t)=\frac{t}{\sinh t}+\sum_{k=0}^n \frac{(2^{2k}-2) B_{2k}}{(2k)!} t^{2k} .
\end{equation}
We perform $2n$ successive integration by parts using the known integral
	\begin{equation}
	    \int x^n \ln x \,\mathrm{d}x = x^{n+1}\left[\frac{\ln x}{n+t}-\frac{1}{(n+1)^2}\right]
	\end{equation}
	until the factor $t^{-2n-1}$ reduces to $t^{-1}$. By induction, we arrive at the expression
		\begin{equation}
		\begin{split}
	    \int_0^{\infty} \frac{\ln t}{t^{2n+2}} \,R(t)\,\mathrm{d}t = &\left.-\sum_{l=0}^{2n} \frac{(2n-l)!}{(2n+1)!} t^{l-2n-1} \left(\ln t + H_{2n+1}-H_{2n-l}\right)\, R^{(l)}(t) \right|_0^{\infty} \\
	    & + \frac{1}{(2n+1)!} \int_0^{\infty} \left(\ln t + H_{2n+1}\right)\, R^{(2n+1)}(t)\,\frac{\mathrm{d}t}{t} ,
	    \end{split}
	\end{equation}
	where the $H_m$'s are harmonic numbers. By inspection, the boundary terms vanish. Then
	\begin{equation}
	\begin{split}
	    \zeta'(2n+1)=&\ln\!\left(\frac{\pi}{2^{1/(2^{2n}-1)}}\right) \zeta(2n+1)\\
	    &\hspace{-14mm}+H_{2n+1}\frac{(-1)^n (2\pi)^{2n}}{(2^{2n}-1)(2n+1)!}\int_0^{\infty}  \frac{\mathrm{d}^{2n+1}}{\mathrm{d}t^{2n+1}}\left(\frac{t}{\sinh t}\right)\,\frac{\mathrm{d}t}{t}\\
	    &\hspace{-14mm}+\frac{(-1)^n (2\pi)^{2n}}{(2^{2n}-1)(2n+1)!}\int_0^{\infty} \ln t \frac{\mathrm{d}^{2n+1}}{\mathrm{d}t^{2n+1}}\left(\frac{t}{\sinh t}\right)\,\frac{\mathrm{d}t}{t}\
	    \end{split}
	\end{equation}
	Comparing the second term with the integral representation of $\zeta(2n+1)$ we find that it is proportional to $\zeta(2n+1)$. Writing the integral in the second term in terms of $\zeta(2n+1)$, we are led to the desired result \eqref{relationship3}.
	
    Substituting the integral representation for the zeta function $\zeta(2n+1)$ given by equation \eqref{zetaoddintegrep1} back into equation \eqref{relationship3} and performing simplifications yield integral representation \eqref{integralrepDzeta} for the derivative of the zeta function at odd positive integers.
\end{proof}

\section{Conclusion}\label{conclusion}
We obtained a finite-part integral representation of the values of the Riemann zeta function at odd positive integers. From the FPI-representation, integral representations for $\zeta(2n+1)$ are obtained, and so with relationships between the zeta function and its derivative at odd positive integers, culminating in integral representations for the derivative of the zeta function at odd positive integers. Whether or not the finite-part integral representation of $\zeta(2n+1)$ will give new insight into the nature of $\zeta(2n+1)$, 
we have demonstrated here that the finite-part integral may provide meaningful representations to functions. The properties of a given function is rarely exhibited by a single expression defining the function. The full array of properties can only be discovered by representing the function in various ways. Thus we have infinite series, infinite product, integral, contour integral, continued fraction, residue representations available in describing a function. Each representation offers a perspective that may not be available from other representations. This motivates the development of new means of representing a given function. Here we have seen that finite-part integral representation may be added to the list of available representations for a given function. We have seen that an FPI-representation is a capsule representation which can be unpacked to provide new representations for functions. More than that, it is possible that FPI-representation may provide new fundamental insights into the nature of functions by studying the fundamental properties of the finite-part integral itself.

\section*{Appendix}
Here we show in reasonable detail that the integral exists in equation \eqref{reprep}. The Hurwitz zeta function $\zeta(2,z)$ is analytic in the real line, in particular at $z=1/2$, which renders it locally integrable in the positive real line. Since it is analytic at $z=1/2$, it is integrable there. We establish the behavior of the integrand at infinity by appealing to the fact that $\zeta(2,z)$ is equal the trigamma function $\psi^{(1)}(z)$ for $\operatorname{Re}z>0$. Using the known asymptotic behavior of $\psi^{(1)}(z)$ for large values of the arguments, we have the asymptotic behavior
\begin{equation}
    \zeta(2,z)\sim \frac{1}{z}+\frac{1}{2 z^2} + \sum_{k=1}^{\infty} \frac{B_{2k}}{z^{2k+1}},\;\; |\operatorname{arg}z|<\pi,\;\; |z|\rightarrow\infty .
\end{equation}
Making the substitution $z=(s+1)/2$, we have
\begin{equation}
    \zeta\!\left(2,\frac{s+1}{2}\right)\sim \frac{2}{s(1+1/s)}+\frac{2}{s^2 (1+1/s)^2}+\sum_{k=1}^{\infty} \frac{2^{2k+1}B_{2k}}{s^{2k+1}(1+1/s)^{2k+1}},
\end{equation}
for arbitrarily large $(s+1)/2$. 

We have been careful not to replace $(s+1)/2$ with $s/2$, doing so would miss out subtle features of the large behavior of the function. We next perform the following expansions
\begin{equation}
	\frac{1}{(1+m)^2}=\sum_{m=0}^{\infty}(-1)^m(m+1)x^{m},\;
	\frac{1}{(1+x)^{2k+1}}=\frac{1}{(2k)!} \sum_{m=0}^{\infty}(-1)^m\frac{ (m+2k)!}{m!} x^m
\end{equation}
valid for $|x|<1$. Substituting these expansions back into the series and performing some simplifications on the first two terms lead to
\begin{equation}
\begin{split}
	\zeta\!\left(2,\frac{s+1}{2}\right)\sim & \frac{2}{s}-2 \sum_{k=1}^{\infty} \frac{(2k-1)}{s^{2k+1}} + 4 \sum_{k=2}^{\infty} \frac{(k-1)}{s^{2k}} \\
	&\hspace{12mm}+ \sum_{k=1}^{\infty}\frac{2^{2k+1} B_{2k}}{(2k)! \,s^{2k+1}} \sum_{l=0}^{k}(-1)^l \frac{(l+2k)!}{l!\, s^l} .
	\end{split}
\end{equation}
We rearrange the double sum and separate the even and odd terms powers of $1/s$. The result is
\begin{equation}
	\begin{split}
	&\sum_{k=1}^{\infty}\frac{2^{2k+1} B_{2k}}{(2k)! \,s^{2k+1}} \sum_{l=0}^{\infty}(-1)^l \frac{(l+2k)!}{l!\, s^l}\\
	&\hspace{4mm}=-\sum_{k=2}^{\infty} \frac{(2k-1)!}{s^{2k}}\sum_{l=1}^{k-1} \frac{2^{2l+1} B_{2l}}{(2k-2l-1)! (2l)!} + \sum_{k=1}^{\infty} \frac{(2k)!}{s^{2k+1}} \sum_{l=1}^k \frac{2^{2l+1} B_{2l}}{(2k-2l)! (2l)!}.
	\end{split}
\end{equation}
Substituting the rearranged double sum back into equation and collecting the odd and even powers yield
\begin{equation}\label{heho}
	\begin{split}
\zeta\!\left(2,\frac{s+1}{2}\right)\sim &\frac{2}{s}-2\sum_{k=1}^{\infty} \left[(2k-1)-(2k)! \sum_{l=1}^k \frac{2^{2l} B_{2l}}{(2l)! (2k-2l)!}\right]\frac{1}{s^{2k+1}} \\
&+ \sum_{k=2}^{\infty}\left[ 4(k-1) - (2k-1)! \sum_{l=1}^{k-1} \frac{2^{2l+1} B_{2l}}{(2l)! (2k-2l-1)}	\right]\frac{1}{s^{2k}}.
\end{split}
\end{equation}

Let us look at each terms separately. The inverse even powers of $s$ turns out to vanish due to the fact that
\begin{equation}\label{sum1}
	\sum_{l=1}^{k-1} \frac{2^{2l+1} B_{2l}}{(2l)! (2k-2l-1)!} =\frac{4(k-1)}{(2k-1)!}.
\end{equation}
To establish this summation identity, we form the sum
\begin{equation}
	\sum_{k=2}^{\infty}\frac{4(k-1)}{(2k-1)!} t^k = 2\sinh\sqrt{t}\left(t\coth\sqrt{t}-\sqrt{t}\right)
\end{equation}
and substitute the expansions
\begin{equation}
	\sinh\sqrt{t}=\sum_{k=0}^{\infty} \frac{t^{k+1/2}}{(2k+1)!},\;\;\; \coth\sqrt{t}=\sum_{l=0}^{\infty} \frac{2^{2l} B_{2l} \,t^{k-1/2}}{(2l!)}
\end{equation}
back into the equation. Performing a rearrangement to collect equal powers of $t$ yields the equality 
\begin{equation}
	\sum_{k=2}^{\infty} \frac{4(k-1)}{(2k-1)!}t^k=\sum_{k=2}^{\infty} \sum_{l=1}^{k-1} \frac{2^{2l+1} B_{2l}}{(2l)! (2k-2l-1)!}t^k .
\end{equation}
By appealing to the uniqueness of power series, the coefficients of each power of $t$ must be equal. This leads to the desired equality \eqref{sum1}.

For the odd powers, we wish now to evaluate the sum in the coefficient. The sum can be written as
\begin{equation}\label{sumo}
	(2k)!\sum_{l=1}^{k}\frac{2^{2l} B_{2l}}{(2l)!(2k-2l)!}=(2k)!\sum_{l=1}^{k-1}\frac{2^{2l} B_{2l}}{(2l)!(2k-2l)!}+2^{2k} B_{2k}
\end{equation}
obtained by isolating the $l=k$ term. We have by the Namias recurrence relation \cite{namias}
\begin{equation}
	B_{2k}=\frac{1}{2(2^{2k}-1)} \left[(2k-1)-(2k)!\sum_{l=1}^{k-1}\frac{2^{2l} B_{2l}}{(2l)!(2k-2l)!}\right].
\end{equation}
The sum in the right hand side of \eqref{sumo} can be ellimanated using thisr recurrence relation. We obtain the sum
\begin{equation}\label{sumo1}
	(2k)!\sum_{l=1}^{k}\frac{2^{2l} B_{2l}}{(2l)!(2k-2l)!}=(2k-1)-(2^{2k}-2) B_{2k}.
\end{equation}

Substituting the sums back into equation \eqref{heho}, we obtain the desired asymptotic expansion 
\begin{equation}
	\zeta\!\left(2,\frac{s+1}{2}\right)\sim \frac{2}{s}-2\sum_{k=1}^{\infty}(2^{2k}-2) B_{2k}\frac{1}{s^{2k+1}}, s\rightarrow\infty .
\end{equation}
This implies the asymptotic equality
\begin{equation}
	\frac{1}{2}s^{2n+1} \zeta\left(2,\frac{s+1}{2}\right)+\sum_{l=0}^n (2^{2l}-2) B_{2l} s^{2n-2l} = O(s^{-2}),\; s\rightarrow \infty,
\end{equation}
for all positive integer $n$. Thus the integral exists at the upper limit of integration. This justifies the interchange of the two integrations above and hence establishes the result. 

\section*{Acknowledgement} This work was funded by the UP-System Enhanced Creative Work and Research Grant (ECWRG 2019-05-R). 

Data sharing is not applicable to this article as no new data were created or analyzed in this study.

\end{document}